\documentclass[12pt,a4paper]{article}
\usepackage[T2A]{fontenc}
\usepackage[russian,english]{babel}
\usepackage[utf8]{inputenc}
\usepackage{amssymb,amsmath,amsthm,amsfonts,bm}

\newtheorem{theorem}{Theorem}

\renewcommand{\div}{\mathop{\rm div}\nolimits}

\newcommand{\norm}[1]{\lVert#1\rVert}
\newcommand{\abs}[1]{\lvert#1\rvert}

\title{Unconditionally stable schemes for non-stationary convection-diffusion equations}
\author{Nadezhda Afanasyeva, Petr Vabishchevich, Maria Vasil'eva}
\date{}

\begin{document}

\maketitle
\begin{abstract}
Convection-diffusion problem are the base for continuum mechanics. 
The main features of these problems are associated with an indefinite operator the problem. 
In this work we construct unconditionally stable scheme for non-stationary convection-diffusion equations, 
which are based on use of new variables. Also, we consider these equations 
in the form of convection-diffusion-reaction and construct unconditionally stable schemes 
when explicit-implicit approximations are used with splitting of the reaction operator.
\end{abstract}

\section{Introduction}
Convection-diffusion equation are basic in the mathematical modelling of the problems of continuum mechanics. 
The main features of these problems are connected with the nonseldadjoint property of elliptic operator 
and domination of convective transport.
When considering compressible media, an operator of convection-diffusion problem is indefinite.
In this case, given process can be nondissipative, i.e. norm of the homogeneous problem solutions does not decrease with time.
This behavior of the norm solutions need to pass on the discrete level in choosing of approximations in time.

In the numerical solution of non-stationary problems for  convection-diffusion  equations the most widely used two- and three-layer scheme.
Investigation of the stability and convergence of approximate solutions can be performed 
using the general theory Samarskii A. A. of stability (correctness) of operator-difference schemes
\cite{Samarskii1989,SamarskiiMatusVabischevich2002}.
Must be kept in mind that for convection-diffusion problems direct application of the general stability criteria can be difficult due to non-selfadjoint operators.
Note also that, in view of indefinite operator of problem we need to oriented $\varrho$-stable ($\varrho > 1$) operator-difference schemes.
In the solution non-stationary problems of long periods of time preference should be given asymptotically 
stable schemes \cite{SamarskiiVabishchevich1995}.
For these schemes ensures the correct behavior of the solutions with the release 
of the fundamental solutions for large time and damping of others.

In this paper, we construct unconditionally stable scheme for the approximate solution of non-stationary convection-diffusion problems.
Such schemes can be applied to other problems with an indefinite operator.
The study conducted by the example of a model two-dimensional boundary-value problem in a rectangle.
Used the simplest approximation of the operators of diffusive and convective transfer on a uniform rectangular grid.
Constructed unconditionally $\varrho$-stable difference scheme based on the introduction of new variables and the explicit-implicit approximation.

\section{The convection-diffusion problem}
We consider the Neumann problem in a rectangle for the non-stationary convection-diffusion equation.
For simplicity, assume that the coefficient of diffusion transport is a constant 
(independent of time, but depends on the point of the computational domain).
The coefficient of convective transport is natural to consider the variables both in space and time.

In the rectangle
\[
\Omega = \{ \bm{x} \ | \ \bm{x} =\left(x_1, x_2\right), \quad 0 < x_\alpha <
l_\alpha, \quad \alpha = 1,2 \}.
\]
We consider the non-stationary convection-diffusion equation with the convective transport in divergent form,
\begin{equation} \label{eq:2-1-convdiff2}
\begin{split}
\frac{\partial u}{\partial t} &
+ \sum_{\alpha = 1}^{2} \frac{\partial}{\partial
x_\alpha}\left(v_\alpha\left(\bm{x}, t\right) u\right)\\ &
- \sum_{\alpha = 1}^{2} \frac{\partial}{\partial x_\alpha} \left( k(\bm{x})
\frac{\partial u}{\partial x_\alpha} \right) 
= f\left(\bm{x}, t\right), \quad \bm{x} \in \Omega, \quad 0 < t \leq T,
\end{split}
\end{equation}
in the standard assumptions $k_1  \leq k\left(\bm{x}\right) \leq k_2$, $k_1 > 0$, $T > 0$ . 
This equation is supplemented by Neumann boundary conditions
\begin{equation} \label{eq:2-1-convdiff-bc}
k\left(\bm x\right) \frac{\partial u\left(\bm{x}, t\right)}{\partial n} = 0,
\quad \bm{x} \in \partial \Omega, \quad 0 < t \leq T.
\end{equation}
For the unique solvability of the nonstationary problem the initial condition is given
\begin{equation} \label{eq:2-1-convdiff-init}
u\left(\bm{x}, 0\right) = u^0(\bm{x}), \quad \bm{x} \in \Omega.
\end{equation}

On the set of functions $u\left(\bm{x}, t\right)$,  which satisfy the boundary conditions (\ref{eq:2-1-convdiff-bc}), non-stationary convection-diffusion problem written in the form of differential-operator equation
\begin{equation} \label{eq:2-1-cd}
\frac{d u}{d t} +\mathcal{ A}u = f(t), \quad \mathcal{ A} = \mathcal{ C}(t) +
\mathcal{ D}, \quad 0 < t \leq T.
\end{equation}
The diffusion operator $\mathcal{ D}$  is defined by
\[
\mathcal{ D}u = - \sum_{\alpha = 1}^{2} \frac{\partial}{\partial x_\alpha} \left(
k\left(\bm{x}\right) \frac{\partial u}{\partial x_\alpha} \right)
\]
and convection operator $\mathcal{ C}$
\[
\mathcal{ C} u = \sum_{\alpha = 1}^{2} \frac{\partial }{\partial x_\alpha}
\left(v_\alpha\left(\bm{x}, t\right) u\right).
\]
Cauchy problem is considered for the evolution equation (\ref{eq:2-1-cd}):
\begin{equation} \label{eq:2-1-cdinit}
u\left(0\right) = u^0.
\end{equation}

For convection operator we have the following representation
\[ 
\mathcal{ C} = \mathcal{ C}_0 + \frac{1}{2} \div \bm{v} \mathcal{E},
\quad 
\mathcal{ C}_0 u = \frac{1}{2} \sum_{\alpha = 1}^{2}
\left(
v_\alpha\left(\bm{x}, t\right) \frac{\partial u}{\partial x_\alpha} + 
 \frac{\partial }{\partial x_\alpha} \left(v_\alpha(\bm{x}, t) u\right)
\right),
\]
where $\mathcal{E}$ --- the identity operator and $\mathcal{ C}_0$
-the operator of convective transport in a symmetric form.

For arbitrary functions $u(\bm{x})$ и $w(\bm{x})$, we define the Hilbert space $\mathcal{ H} = L_2\left(\Omega\right)$ 
with inner product and norm
\[
\left(u, w\right) = \int_\Omega u\left(\bm{x}\right) w\left(\bm{x}\right) dx, 
\quad
\norm{u} = \left(u, u\right)^{1/2}.
\]

Diffusion operator $\mathcal{ D}$  on the set of functions satisfying (\ref{eq:2-1-convdiff-bc}), is self-adjoint and positive define
\begin{equation} \label{eq:2-1-doperator}
\mathcal{ D} = \mathcal{ D}^*  \geq 0.
\end{equation}
The operator of convective transport is considered under the assumption that the normal component of the medium velocity $\bm{v} = (v_1,v_2)$ on the boundary is zero:
\begin{equation} \label{eq:2-1-convnepro}
v_n(x) = \bm{v} \, \bm{n} = 0, \quad x \in \partial \Omega,
\end{equation}
where $\bm{n}$ --- outward normal to $\partial \Omega$.
In $\mathcal{ H}$ the convection operators have the following properties:
\begin{equation} \label{eq:2-1-coperator}
\mathcal{ C}_0 = -\mathcal{ C}^*_0 .
\end{equation}
Also useful upper estimates for convective transport operator $\mathcal{ C}$:
\begin{equation} \label{eq:2-1-coperator2}
\abs {\left(\mathcal{ C} u, u\right)} \leq \delta \norm{u}^2, 
\quad
\delta = \frac{1}{2} \norm{\div \bm{v}}_{C\left(\Omega\right)}.
\end{equation}

\section{The differential-difference problem}
For an approximate solution of the non-stationary convection-diffusion problem we use a uniform grid in the area $\Omega$:
\[
\omega = \{ \bm{x} \ | \ \bm{x} =\left(x_1, x_2\right), \quad x_\alpha =
\left(i_\alpha + \frac{1}{2}\right) h_\alpha, 
\]
\[
\qquad i_\alpha = 0,1,...,N_\alpha,
\quad (N_\alpha + 1) h_\alpha = l_\alpha, \quad \alpha = 1,2 \} .
\]
We define the Hilbert space $H=L_2\left(\omega\right)$ for grid functions, where the inner product and norm are defined as follows:
\[
\left(y, w\right) \equiv  \sum_{\bm x \in  \omega} y\left(\bm{x}\right)
w\left(\bm{x}\right) h_1 h_2,  \quad 
\norm{y} \equiv  \left(y, y\right)^{1/2}.
\]

For the difference operator of the diffusion transfer $D$ is used additive representation
\begin{equation} \label{eq:2-1-appd}
D = \sum_{\alpha=1}^{2} D^{\left(\alpha\right)}, 
\quad \alpha = 1, 2, \quad \bm{x} \in \omega,
\end{equation}
here $D^{\left(\alpha\right)}, \ \alpha=1,2$ is associated with the corresponding differential operator in one direction.

The difference operator of diffusion transport (\ref{eq:2-1-appd}) in $H$ is self-adjoint and positive definite \cite{Samarskii1989}
\begin{equation} \label{eq:2-1-dappoperator-1}
D = D^* \geq 0.
\end{equation}

The convective terms are approximated with second-order, using the central difference derivatives and shifted grids to specify the velocity components.
For the difference operator of convective transport are also using additive representation
\begin{equation}
C y = \sum_{\alpha=1}^{2} C^{\left(\alpha\right)} .
\end{equation}

For the case of sufficiently smooth velocity components and solutions of the differential problem, for example, we can assume,
\[
  b_\alpha({\bm x}) = v_\alpha({\bm x}), 
  \quad {\bm x} \in \Omega, 
  \quad 0 <  x_\alpha < l_\alpha,
\] 
\[
  b_\alpha({\bm x}) = 0, 
  \quad x_\alpha = 0,
  \quad x_\alpha = l_\alpha,
  \quad \alpha=1,2 .
\]

Difference operator of convective transport in symmetric form have the following basic property:
\begin{equation} \label{eq:2-1-cappoperator-1}
C_0^* = -C_0.
\end{equation}
We also \cite{SamarskiiVabishchevich1995} have the grid analogue of inequality (\ref{eq:2-1-coperator2}):
\begin{equation}\label{eq:2-1-cappoperator-2}
\abs{\left(C y, y\right)} \leq \delta \norm{y}^2
\end{equation}
with a constant
\[
\delta = \frac{1}{2} \max_{\bm x \in \omega} \left| \frac{
b^{(1)}\left(x_1+0.5h_1,x_2\right) - b^{(1)}\left(x_1-0.5h_1,x_2\right)  }{h_1} 
\right.+
\]\[ \left.\qquad + \frac{ b^{(2)}\left(x_1, x_2+0.5h_2\right) -
b^{(2)}\left(x_1, x_2-0.5h_2\right) }{h_2} \right| .
\]

Therefore, from the equation (\ref{eq:2-1-cd}) we arrive at the differential-operator equation
\begin{equation} \label{eq:2-1-gridcd}
\frac{d y}{dt} + A y = \phi\left(t\right), \quad A = A\left(t\right) = C + D,
\quad 0 < t \leq T,
\end{equation}
on the set of grid functions $y(t) \in H$ with the initial condition
\begin{equation} \label{eq:2-1-gridinit}
y\left(0\right) = y^0.
\end{equation}
Difference  convection and diffusion operators in the differential-difference problem inherit the basic properties of differential operators.

\section{Unconditionally stable schemes}
For simplicity, we restrict ourselves to a uniform grid in time
\[
  \bar\omega_\tau =
  \omega_\tau\cup \{T\} =
  \{t^n=n\tau,
  \quad n=0,1, ... ,N,
  \quad \tau N =T\} .
\]
For an approximate solution of (\ref{eq:2-1-gridcd}), (\ref{eq:2-1-gridinit}) commonly used two-layer weighted scheme, 
which have a following restrictions on the time step $\tau < \tau_0 = \frac {1}{\sigma \delta}$.

To construct the unconditionally stable schemes for the solution 
of the differential problem (\ref{eq:2-1-gridcd}), (\ref{eq:2-1-gridinit}) 
with $A \geq - \delta E, \ \delta > 0$ we define a new function  $w$:
\begin{equation} \label{41}
y = \exp(\delta t) w.
\end{equation}
Substitution of (\ref{41}) in (\ref{eq:2-1-gridcd}), (\ref{eq:2-1-gridinit}) 
with homogeneous right-hand side gives the following problem for the $w$:
\begin{equation} \label{42}
\frac{d w}{dt} + \tilde{A} w = 0, \quad \tilde{A} = A + \delta E, 
\quad 0 < t \leq T,
\end{equation}
\begin{equation} \label{43}
w\left(0\right) = y^0.
\end{equation}
Under this transformation, operator $\tilde{A}$ is a nonnegative ($\tilde{A} \geq 0$).

To solve the problem (\ref{42}), (\ref{43}) we use a two-layer weighted difference scheme, 
which is unconditionally stable for standard restrictions $\sigma \geq 0.5$.
We write the scheme for the grid function $y^n$
\begin{equation} \label{eq:newscheme-new}
\frac{\exp(-\delta \tau) y^{n+1} - y^n}{\tau} + \left( A + \delta E\right) \left(
\sigma \exp(-\delta \tau) y^{n+1} + \left(1 - \sigma\right) y^n \right) = 0,
\end{equation}
\begin{equation} \label{eq:newscheme-newinit}
y^0 = u^0, \quad
t^n \in \omega_\tau,.
\end{equation}
In contrast to the non-standard schemes considered in \cite{Mickens2002}, the positive effect is achieved not only through the use of a new approximation of the time, but also by correcting the problem operator.

\begin{theorem}
The difference scheme (\ref{eq:newscheme-new}), (\ref{eq:newscheme-newinit}) with  $\sigma \geq 0.5$  unconditionally $\varrho$-stable in $H$ with
\begin{equation} \label{eq:2-1-ns-3}
\varrho = \exp(\delta \tau) ,
\end{equation}
with the a priori estimate for solutions
\begin{equation} \label{eq:2-1-ns-2}
\norm{y^{n+1}} \leq \varrho \norm{y^n} .
\end{equation}
\end{theorem}

\begin{proof}
We rewrite the scheme (\ref{eq:newscheme-new}),
(\ref{eq:newscheme-newinit}) in form
\begin{equation} \label{eq:2-1-ns-1}
\frac{\exp(-\delta \tau) y^{n+1} - y^n}{\tau} + \tilde{A} p^{n+1}= 0, \quad t^n \in
\omega_\tau,
\end{equation}
where
\[
p^{n+1} = \sigma \exp(-\delta \tau) y^{n+1} + \left(1 - \sigma\right) y^n =
  \tau \left( \sigma - \frac{1}{2} \right) r^{n+1} +
  \frac{1}{2} \left (\exp(-\delta \tau) y^{n+1} - y^n \right ),
\]
\[
r^{n+1} = \frac{\exp(-\delta \tau) y^{n+1} - y^n}{\tau}.
\]
Multiplying the scalar equation (\ref{eq:2-1-ns-1}) by $p^{n+1}$, we obtain the equality
\[
\tau \left( \sigma - \frac{1}{2} \right) \left(r^{n+1}, r^{n+1} \right) + \tilde{A}
\left(p^{n+1}, p^{n+1}\right)
\]
\[
+ \frac{1}{2 \tau}\left( \left(\exp(-\delta \tau)
y^{n+1}, \exp(-\delta \tau) y^{n+1}\right) - \left(y^n, y^n\right) \right) = 0
\]
From this equation, under the condition $\sigma \geq 0.5$ and $\tilde{A} \geq 0$, yields the estimate of stability (\ref{eq:2-1-ns-2}),(\ref{eq:2-1-ns-3}).
\end{proof}

Equation (\ref{eq:2-1-convdiff2}) can be written in the form 
of convection-diffusion-reaction equation with the convective terms in the symmetric form
\begin{equation}\label{eq:cdr}
\begin{split}
   \frac{\partial u}{\partial t} & + {C}_0 u + {D} u + {R} u = f({\bm x},t),
   \quad {\bm x}\in \Omega,
   \quad t > 0.
\end{split}
\end{equation}
where
\[
  { R} u = r({\bm x},t) y , \quad r({\bm x},t)  = \frac{1}{2} \div {\bm v} .
\] 
For the reaction operator we have the estimate
\begin{equation}\label{eq:r}
  { R} = { R}^* ,
  \quad -\delta  {\cal E} \leq { R} \leq \delta { E}.
\end{equation}

To construct the unconditionally stable scheme without the assumption of nonnegativity operator of problem we will use the explicit-implicit approximation for the equation (\ref{eq:cdr}). 
The problem is generated by the reaction operator therefore we split it into two:
\begin{equation}\label{3-1.37}
  R = R_+ + R_-,
  \quad R_+ = R_+^*,
  \quad R_- = R_-^*,
  \quad 0 \leq R_+ \leq \delta E,
  \quad -\delta E \leq R_- < 0 .   
\end{equation} 

When using the two-layer explicit-implicit schemes, we can only count on first-order accuracy in time.
Therefore it is natural oriented to purely implicit approximation of the basic terms of the operator 
and define following difference scheme
\begin{equation}\label{eq:cdrsh}
  \frac {y^{n+1} - y^n}{\tau} +
  (C^n + D + R_+^n) y^{n+1} + R_-^n y^n = 0,
  \quad n = 0,1,...,N-1 .
\end{equation}

\begin{theorem}\label{3-1.t-4}
Explicit-implicit difference scheme  (\ref{eq:cdrsh}), (\ref{eq:newscheme-newinit})  unconditionally $\varrho$-stable in $H$ with
\begin{equation}\label{3-1.39}
  \varrho = 1 + \delta \tau 
\end{equation} 
for the numerical solution we have the estimate:
\begin{equation}\label{3-1.40}
  \|y^{n+1}\| \le \varrho \|y^n\|,
  \quad n = 0,1,...,N-1 .  
\end{equation} 
\end{theorem}

It is important to note that, in contrast to the ordinary weighted scheme, stability is obtained without restrictions on the time step.
The transition to a new time layer associated with the solution of the grid problem
\begin{equation}\label{eq:newscheme-slae}
  (E +  \sigma \tau (A + \delta E ))  y^{n+1} =  \chi^n
\end{equation}
for scheme (\ref{eq:newscheme-new}) and 
\begin{equation} \label{3-1.slau}
 (E + \tau (C + D + R_+)) y^{n+1} = r^n
\end{equation}
for scheme (\ref{eq:cdrsh}).

Equation (\ref{eq:newscheme-slae}) and (\ref{3-1.slau})  is a system of linear algebraic equations with a positive definite nonselfadjoint matrix.
For solutions that can apply the standard iterative methods.

\end{document}